\documentclass[12pt]{amsart}
\usepackage{amsfonts}
\usepackage{amssymb}
\usepackage{pict2e}
\font\got=eufm10

\def\g2{ \hbox{\got g}_2}
\def\e7{ \hbox{\got e}_7}
\def\e8{ \hbox{\got e}_8}
\def\id{\mathop{\hbox{\rm id}}}
\def\h{\hbox{\got h}}

\def\L{{\mathcal L}}
\def\J{{\mathcal J}}

\def\a{\alpha}
\def\s{\sigma}
\def\tr{\mathop{\rm tr}}
\def\T{\mathcal{T}}

\def\A{\mathcal A}
\def\CC{\mathcal C}
\def\X{\mathfrak{X}}\def\supp{\mathop{\hbox{\rm  Supp}}}
\def\M{\mathcal M}

\def\span#1{\langle#1\rangle}
\def\esc#1{\langle#1\rangle}
\def\H{{\mathbb  H}}

\def\OO{\mathbb{O}}
\def\C{\mathbb C}
\def\Z{\mathbb Z}
\def\F{\mathbb F}
\def\R{\mathbb R}

\newcommand{\ii}{\textbf{i}}
\def\aut{\mathop{\rm Aut}}

\def\der{\mathop{\rm Der}}

\def\Ad{\mathop{\rm Ad}}

\def\ad{\mathop{\rm ad}}

\def\sll{\mathop{\rm sl}}
\def\fix{\mathop{\rm fix}}
\def\dim{\mathop{\rm dim}}

\def\sp31{\mathop{ \mathfrak{sp}}_{3,1}(\mathbb{H})}
\def\sll3{\mathop{ \mathfrak{sl}}(3,\mathbb{O})}

\def\Sp{\mathop{\rm Sp}}

\def\Mat{\mathop{\rm Mat}}

\def\a{\alpha}
\def\si{\sigma}
\def\sign{\mathop{\hbox{\rm sign}}}
\def\I{\mathbf{i}}

\def\rank{\mathop{\rm rank}}
\def\tr{\mathop{\rm tr}}

\def\diag{\mathop{\rm diag}}
 \def\e6{  \mathfrak{e}_6 }
 \def\f4{  \mathfrak{f}_4 }
 \def\c4{  \mathfrak{c}_4 }
 \def\sig#1{ \mathfrak{e}_{6,-#1} }

\newtheorem{defi}{Definition}
\newtheorem{te}{Theorem}
\newtheorem{pr}{Proposition}
\newtheorem{lm}{Lemma}
\newtheorem{co}{Corollary}
\newtheorem{re}{Remark}
\title[Fine gradings on $\sig{26}$  ]{ Symmetries on the real form $\sig{26}$}  
\author[C. Draper]{Cristina Draper${}^\star$ }
\thanks{${}^\star$   Partially supported by MCYT grant   MTM2013-41208-P  and by   the Junta de Andaluc\'{\i}a PAI
projects  FQM-336,  FQM-7156.}
\address{Cristina Draper Fontanals: Departamento de
Matem\'atica Aplicada\\ Escuela de las Ingenier\'{\i}as\\ Ampliaci\'{o}n Campus de Teatinos, S/N, 29071 M\'alaga,
Spain\\cdf@uma.es}
\author[V. Guido]{Valerio Guido${}^\dagger$ }
\thanks{${}^\dagger$   Assegno di Ricerca in Algebra presso l'Universit\`{a} del Salento - 
Dipartimento di Matematica, approvato con Decreto n.197 del 03/12/2014}
\address{Valerio Guido: Dpto. di Matematica E. De Giorgi\\ Universit\`{a} del Salento\\
Via Provinciale Lecce-Arnesano, 73100, Lecce, Italy\\valerio.guido@unisalento.it}



\usepackage{graphicx}
\begin{document}

\setlength{\unitlength}{0.06in}

\maketitle


\begin{abstract}
We describe   four fine gradings on the real form $\sig{26}$. They are precisely the gradings whose complexifications are fine gradings on the complexified algebra $\mathfrak{e}_6$. The universal grading groups are $\Z_2^6$, $\Z\times\Z_2^4$, $\Z^2\times\Z_2^3$, and  $\Z_4\times\Z_2^4$.
\end{abstract}




\section{Introduction}

Gradings of Lie algebras  have appeared in mathematical physics \cite{fisicamuyref} and particularly in particle physics \cite{fisicos}.
One of the keys of this interplay is that each fine grading on a Lie algebra provides a maximal set  of additive quantum numbers (see, for instance, \cite{numeroscuanticos}). There also seems that the exceptional Lie algebra $\mathfrak{e}_6$   could play some role in particle physics:
according to  \cite[Chapter~27]{libroGeorgi}, the  search for unified theories in particle physics leads to a theory based on the algebra $\mathfrak{e}_6$: it is a natural step as a candidate gauge group for a Grand Unified Theory, in the progression $SU(5)$, $SO(10)$. Also recall that the group $E_6$ has a long history of applications in physics from the apparition of the   Jordan algebras, more precisely the Albert algebra, as a search of a convenient formalism of the quantum Mechanics \cite{jordanalgebras}. 

Gradings are necessary too for the new theory of graded contractions of Lie algebras. Originally the notion of contraction implicitly appeared in physics as a change of symmetry, usually connected with certain asymptotic limits of physical theories (e.g. see \cite{IW53} where the Galilei group appears as the limit of the Lorentz group as the speed of light tends to infinity). In the modification of Wigner's approach posed by Patera and his collaborators \cite{contraccionesgraduadas}, the contractions preserve a fixed grading by an abelian group.  

 {For any finite dimensional simple Lie algebra   over $\C$, the root decomposition is a fine (group) grading which provides   basic information about the structure of the algebra whose importance
for applications is difficult to overestimate. This is a first reason to study \emph{fine} gradings (those ones which cannot be splitted in  smaller pieces), the possible alternative approaches to the structure theory and
representation theory.  This study  was initiated by \cite{LGI}, from the belief of unexplored applications. First, every grading is  obtained from one particular fine grading. Second, the study of fine gradings provides alternatives to the
Chevalley basis of semisimple Lie algebras: some problems could be   naturally and more simply
formulated exploiting bases dictated by other fine gradings   than the root
decomposition, as happens for instance  with the second fine grading of $\mathfrak{sl}(2, \C)$,
which is a $\Z_2\times\Z_2$-grading  spanned by the Pauli's matrices. Also the generalization  of
the Pauli's matrices to  $\mathfrak{gl}(n, \C)$ is finding its way into the physics literature \cite{aplicacioneslasPauli}.
 Third, the study of deformations of Lie algebras during
which a chosen grading is preserved seems to be very
useful in applications \cite{contraccionesgraduadas}.  Fourth the problem of finding isomorphism classes of solvable Lie algebras
of a given dimension is reduced in \cite{PatZas}  to the classification of isomorphism classes
of equidimensional nilpotent Lie algebras with the help of fine gradings. More directions and applications of fine gradings can be found in the recent monograph about gradings on simple Lie algebras \cite{libro}. }

There is not much work developed about fine gradings on  real forms of simple Lie algebras. While the classification of the fine gradings on the complex finite-dimensional simple Lie algebras is almost complete \cite{libro}, the same problem for real algebras has only be treated for $\mathfrak{g}_2$ and $\mathfrak{f}_4$ in \cite{reales} and for some classical Lie algebras of low-dimension (for instance, in \cite{LGIII}). 

In this paper we begin the study of the gradings on the real forms of  $\e6$, the \emph{next} exceptional Lie algebra (i.e., the left one with the least possible dimension).
This complex Lie algebra  has five real forms, characterized by the  signatures of their Killing forms, namely, $6$, $2$, $-14$, $-26$ and $-78$.
We focus on that one of signature -26 because of its multiple apparitions related to some other objects in Physics and Mathematics. First, $E_{6,-26}$ is described as the group of symmetries of the Albert algebra, in the sense that these symmetries preserve the determinant (not the product) of this 27-dimensional exceptional Jordan algebra.
Besides, $E_{6,-26}$ is just the group of collineations (line-preserving transformations) of the projective plane $\OO\mathbb{P}^2$ (see \cite{Baez} for more information about this real form), sometimes described as the symmetry group of the bioctonionic projective plane \cite{Rosenfeld}. A description of the group $E_{6,-26}$ as $SL(3,\OO)$ is given in \cite{octonionesye6}, generalizing the interpretation of $SL(2,\OO)$ as the double cover of the generalized Lorentz group $SO(9,1)$. 
These facts have provoked quite recent attention on our real form: for instance, \cite{Wangberg} puts its attention on remarkable subalgebras, and Manogue and Drey apply $\sig{26}$ to   Particle Physics \cite{octonionesye6yfisparticulas},
  describing  many properties of leptons in a natural way only by choosing     one of the octonionic directions and  one of the $2\times 2$ submatrices inside the $3\times 3$ matrices (the Albert algebra) to be special.  

In order to emphasize the general relevance of the exceptional group $E_6$ in Mathematics, we will mention   a couple of additional examples. The groups of type $E_6$ have appeared recently   in Differential Geometry, in order to answer what is known as \emph{irreducible holonomy problem:}   which groups can occur as holonomies of torsion-free affine connections \cite{masholonomia}. The complex group $E_6$ as well as the real groups $E_{6,6}$ and $E_{6,-26}$ appear  in the list of \emph{exotic holonomies}, which are defined to be those ones missing during time. That work was continued by \cite{holonomiae6}, which proved that   if a torsion-free affine connection has holonomy contained in $E_6$ (in any signature) and is also Ricci-flat, then it is flat.     
Also algebraic geometers are interested in  the group $E_6$: for instance, the automorphism group of the configuration of the 27 lines on a smooth cubic surface in
$\mathbb{CP}^3$
 can be identified with the Weyl group of $\e6$, and such 27 lines on the cubic surface are in natural correspondence with the weights of the minimal
representation of $E_6$ \cite{26lineas}.  \smallskip

The structure of the paper is as follows. First, some preliminaries about real forms are exposed, stressing   some methods to determine lately the signature relating it with order two automorphisms of the complexified algebra. Section~3 is devoted to develop models of $\sig{26}$  which afterwards will be adapted to the gradings. This makes necessary     to recall the famous Freudenthal's magic square related to the Tits' construction and some basic facts about composition and Jordan algebras. The background on gradings is compiled in Section~4, which also contains a very brief   sketch of the classification results on gradings on the complex Lie algebra $\e6$ as well as some of the main  methods used in the study of   gradings on  real forms. Our main results are presented in Section~5. Namely, fine gradings on $\sig{26}$ over the groups $\Z_2^6$, $\Z\times\Z_2^4$, $\Z^2\times\Z_2^3$  and  $\Z_4\times\Z_2^4$ are described,   followed by a proof that these gradings exhaust all the possible cases of fine gradings inherited from $\e6$. We finish with a list of   conclusions and open problems.

\smallskip


\section{Preliminaries on real forms  }\label{sec_prelimirealforms}

The material about  real forms is extensively developed in the book  \cite{libroreales}, although here we follow the approach in \cite{apuntesAlb}.

If $L$ is a real Lie algebra,  $L$ is said a \emph{real form} of a complex Lie algebra $S$ if the complexification $L^\C\cong S$,
where the bracket in    $L^\C=L \otimes_{\R}\C\cong L\oplus \ii L $ is defined by
$$
[x_1+\ii y_1, x_2+\ii y_2] = [x_1,x_2]-[y_1,y_2]+ \ii ([x_1,y_2] + [y_1,x_2])
$$
for all $x_1,x_2,y_1,y_2 \in L$.   
If $S=L^\C$, the map  $\sigma(x+\ii y)=x-\ii y$ is a \emph{conjugation} (conjugate-linear order two map) such that $L= S^\sigma:=\{x\in S\mid \sigma(x)=x\}$.
 Two different real forms $ S^{\sigma_1}$ and $ S^{\sigma_2}$ are isomorphic if and only if  $\sigma_1$ and $\sigma_2$ are conjugate, that is, there is $\alpha\in\aut(S)$ such that $\sigma_2=\a\sigma_1\a^{-1}$.

The real forms of a complex semisimple Lie algebra $S$ are characterized by the signatures of their   Killing forms. Recall that, for $L$ a real Lie algebra, 
 the \emph{Killing form} of $ L$ is the symmetric bilinear form $k\colon L\times  L \rightarrow \R$ defined by
$
k(x,y)=\tr (\ad x \ad y) , 
$
where $\ad x (y)=[x,y]$  for any $x,y\in  L$. 
If $L$ is semisimple, by Cartan's criterion the Killing form is nondegenerate and can be diagonalized in a suitable basis with the diagonal entries $\pm1$. The signature of $k$ is defined as $n_+-n_-$, where $n_{\pm}$ is the number of $\pm1$.
The real form $L$ is called \emph{compact} if the Killing form is (negative) definite, so it is characterized by having signature equal to $-\dim S$. The conjugation $\tau$ such that $L=S^\tau$ is also called compact. 
The real form $L$ is called \emph{split} if it contains a Cartan subalgebra $\mathfrak{h}$ such that $\ad(h)$ is diagonalizable (over $\R$) for any $h\in \mathfrak{h}$. In this case the signature of $L$ coincides with the rank of $S$. Any complex semisimple Lie algebra $S$ possesses both a compact   and a split real form (by the above,  unique up to isomorphisms).  We recall for later use how to construct these forms (extracted from \cite[Appendix A]{apuntesAlb}).   For any $\a_j$ in a basis $ \{\a_1,\dots,\a_n\}$ of the root system $\Delta$ of $S$ relative to a Cartan subalgebra ($n=\rank S$), take $e_j\in S_{\a_j}$ and $f_j\in S_{-\a_j}$ such that $\a_j(h_j)=2$ for $h_j=[e_j,f_j]$.
 For any $\alpha\in\Delta^+$ choose $1\le j_1,\dots,j_m\le n$ such that $\alpha_{j_{1}}+\dots+\alpha_{j_{m}}=\alpha$ and $\alpha_{j_{1}}+\dots+\alpha_{j_{i}}\in\Delta$ for all $i\le m$. 
Denote   
  $e_\alpha:= [e_{j_{m}},[e_{j_{m-1}},\dots [e_{j_{2}},e_{j_{1}}]\dots]]$
  and $f_\alpha:= [f_{j_{m}},[f_{j_{m-1}},\dots [f_{j_{2}},f_{j_{1}}]\dots]]$. Then the conjugation $\sigma_0$ of $S$ fixing the basis 
 $\mathcal{B}= \{h_j\mid j=1,\dots,n\}\cup\{e_\a,f_\a\mid \a\in\Delta^+\}$
 is split, since the  basis $\mathcal{B}$ has  rational numbers as structure constants.
Moreover, if we consider  the (only) automorphism  $\omega\in\aut(S)$ determined by $\omega(e_j)=-f_j$ and $\omega(f_j)=-e_j$ for any $j=1,\dots,n$, then $\sigma_0\omega$ is compact.

 It will be useful for our purposes that the classification of the real forms of $S$ is equivalent to the classification of the order two automorphisms of $S$.
 More precisely, 
 for any $\si$ conjugation of $S$ there exists $\theta_\si\in\aut(S)$ (not unique)  commuting with $\si$ such that the conjugation $\theta_\si\si$ is compact, and the map 
 \begin{equation}\label{eq_defdeFI}
 \begin{array}{rl}\vspace{2pt}\Phi\colon&\{\textrm{Conjugacy classes of conjugations of $S$}\} \\
 &\quad\longrightarrow\{\textrm{Conjugacy classes of order 2 automorphisms of $S$}\} 
 \end{array}
 \end{equation}
given by $\Phi([\si])=[\theta_\si]$, is well defined and bijective. (Here order two means $\theta_\si^2=1$, that is, the identity is included.)
Furthermore, the signature of  the Killing form of $L=S^\si$ coincides with 
\begin{equation}\label{eq_signapartirparte fija}
\dim S-2\dim\fix(\theta_\si).
\end{equation}
 Indeed, the automorphism $\theta_\si$   produces a $\Z_2$-grading on $L=S^\si=L_{\bar0}\oplus 
 L_{\bar1}$ such that  $ L_{\bar0}\oplus \ii
 L_{\bar1}$ is a compact real form. 
 (The restriction of $\theta_\si$ to $L$ is usually called a \emph{Cartan involution} of $L$ and the decomposition $L=L_{\bar0}\oplus 
 L_{\bar1}$ a \emph{Cartan decomposition} of $L$.)
 Since $L_{\bar0}$ and $L_{\bar1}$ are orthogonal relative to $k$ the Killing form of $S$, then $\sign k_L=\sign k\vert_{L_{\bar0}}+\sign k\vert_{L_{\bar1}}$.
 (We are denoting by $k_L$ the Killing form of $L$, which coincides with $k\vert_L$, and by $k\vert_V$ the restriction $k\vert_{V\times V}$ for $V$ a subspace of $L$.)
 But $\sign k\vert_{L_{\bar1}}=-\sign k\vert_{\ii L_{\bar1}}$
 and $k\vert_{L_{\bar0}\oplus \ii
 L_{\bar1}}$ is negative definite, so that $\sign k_L=-\dim_\C S_{\bar0}+\dim_\C S_{\bar1}=\dim_\C S-2\dim_\C S_{\bar0}$.

 \medskip

The classification of the order two automorphisms of a complex Lie algebra was completed by Kac \cite{Kac} (Cartan in the inner case \cite{CartanAutomorfismos}).
Such automorphisms are characterized by the isomorphy class of their fixed subalgebra, and in    case $S$ is  of type $E_6$, by the dimension of such fixed subalgebra. This correspondence for the case $E_6$ is detailed in Table \ref{tabladeautomorfysignaturas}.

\begin{table}[h]
\begin{tabular}{|c||ccccc|}
\hline  
\vrule width 0pt height 12pt
    $\sign  k\vert_{S^{\sigma}}$&$-78 $&$2 $&$-14 $&$-26 $&$6 $\\
   \vrule width 0pt height 12pt
   $\dim\fix(\theta_{\sigma})$&$78 $&$ 38$&$ 46$&$ 52$&$36 $\\
 \vrule width 0pt height 12pt
 $ \fix(\theta_{\sigma})$& $\quad E_6\quad  $ & $\quad A_5+A_1\quad $ & $\quad D_5+Z \quad $&$\quad F_4\quad  $&$\quad C_4\quad  $ \\
 \hline  
 \end{tabular}\vspace{3pt}
\caption{Automorphisms versus signatures}\label{tabladeautomorfysignaturas}
 \end{table}

\section{Models of $\sig{26}$}

\subsection{Composition algebras and Jordan algebras}

The contents of this subsection and the next one are extracted from \cite{Schafer}.
The ground field $\F$ will be always assumed to be either $\R$ or $\C$.

A \emph{Hurwitz algebra} over $\F$   is a  unital algebra $\CC$ endowed with a nonsingular quadratic form $n\colon \CC\to \F$ admitting composition, that is, $n(xy)=n(x)n(y)$. This form $n$ is usually called the \emph{norm}.
Each element $a\in \CC$ satisfies a quadratic equation
$
a^2-t_\CC(a)a+n(a)1=0,
$
where $t_\CC(a)=n(a+1)-n(a)-1$ is called the \emph{trace}. Denote by $\CC_0=\{a\in \CC\mid t_\CC(a)=0\}$ the subspace of traceless elements. Note that $[a,b]=ab-ba\in \CC_0$ for any $a,b\in \CC$, since $t_\CC(ab)=t_\CC(ba)$.
The map $-\colon \CC\to \CC$ given by $\bar a=t_\CC(a)1-a$ is an involution (order 2 antihomomorphism) and $n(a)=a\bar a$ holds. We will need 
 the fact that for any $a,b\in \CC$, the endomorphism
$
d_{a,b}:=[l_a,l_b]+[l_a,r_b]+[r_a,r_b]
$
is a derivation of $\CC$, where $l_a(b)=ab$ and $r_a(b)=ba$ denote the left and right multiplication operators respectively.

There are Hurwitz algebras only in dimensions $1$, $2$, $4$ and $8$. Over the complex numbers there exists just one Hurwitz algebra of each dimension, while there are 7 real Hurwith algebras: $\R$, $\R\times \R$, $\C$, $\H$, $\Mat_{2\times2}(\R)$, $\OO$ and the split octonions $\OO_s$ (Zorn algebra). The most useful choices  for our purposes are:
\begin{itemize}
\item  $\R\times \R$, with  componentwise product and norm given by $n((a,b))=ab$;
\item  The   octonion algebra $\OO$,  which is the real division algebra with basis  
$$
\{1,\textbf{i},\textbf{j},\textbf{k},\textbf{l},\textbf{il},\textbf{jl},\textbf{kl}\},
$$
for the product in  $\langle\{1,\textbf{i},\textbf{j},\textbf{k}\}\rangle=\H$ that one in the quaternion algebra, and 
$$q_1(q_2\textbf{l})=(q_2q_1)\textbf{l},\quad (q_1\textbf{l})(q_2\textbf{l})=-\bar q_2q_1, \quad (q_2\textbf{l})q_1=(q_2\bar q_1)\textbf{l}
$$
 for any $q_i\in\H$.
The norm is determined by $n(\H,\textbf{l})=0$, $n(\textbf{l})=1$ and $n\vert_\H$, which coincides with the usual norm of the quaternion algebra.
\end{itemize}\smallskip

A \emph{Jordan algebra} over $\F$ ($\R$ or $\C$) is a commutative  algebra satisfying the Jordan identity
$
(x^2y)x=x^2(yx)
$.
If $A$ is an associative algebra and the multiplication is  denoted by juxtaposition, then $A^+=(A,\cdot)$ is a Jordan algebra, where  the new product $\cdot$ on $A$ is given by
$
x\cdot y=\frac12(xy+yx).
$
If $\CC$ is a  Hurwitz algebra (with involution denoted by $-$), the algebra of $\gamma$-hermitian matrices of order 3 with respect  to the involution given by $x^*=\gamma\bar x\gamma^{-1}$ ($\gamma=\diag\{\gamma_1,\gamma_2,\gamma_3\}$),
  $$ J=\mathcal{H}_3(\CC,\gamma)=\{x \in \textrm{Mat}_{3\times 3}(\CC)\mid x^*=  x \},$$  
  is a Jordan algebra  for the product $\cdot$ given as above.
Consider the normalized trace $t_J\colon J\to \F$ given by $t_J(x)=\frac{\tr(x)}{3}=\frac{\sum_{i=1}^3 x_{ii}}{3}$ if $x=(x_{ij})\in J$. This   is the only linear map satisfying that $t_J(I)=1$ ($I$ the identity matrix of order 3) and $t_J((x\cdot y)\cdot z)=t_J(x\cdot(y\cdot z))$ for any $x,y,z\in J$. Thus we have a decomposition $J=\F I\oplus J_0$, for $J_0=\{x\in J\mid t_J(x)=0\}$, since $x*y:=x\cdot y-t_J(x\cdot y)I\in J_0$. In particular we have a commutative multiplication $*$ defined in $J_0$.
Denote by $R_x\colon J\to J$, $y\mapsto y\cdot x$  the multiplication operator, and observe that
   $[R_x,R_y]\in\der(J)$
  for any $x,y\in J$.

\subsection{  Tits' construction}\label{subsec_TitsModel}

In 1966, Tits provided a beautiful unified construction of all the exceptional simple Lie algebras
\cite{Tits}.  
 When in this construction we use  a composition algebra $\CC$ and a
 Jordan algebra consisting  of $3\times 3$-hermitian matrices over a second composition algebra $\CC'$,
Freudenthal's magic square   is obtained:\smallskip

\begin{center}
 {\small
 \begin{tabular}{c||cccc|}
 $\dim\CC/\dim\CC'$ & $1$& $2$& $4$& $8$\\
\hline\hline \vrule width 0pt height 14pt
 $1$& $A_1$&$A_2$&$C_3$&${ F_4}$\\
 $ 2$& $A_2$&$A_2\oplus A_2$&$A_5$&${ E_6}$\\
 $ 4$& $C_3$&$A_5$&$D_6$&${ E_7}$\\
 $8$& ${ F_4}$&${ E_6}$&${ E_7}$&${ E_8}$\\\hline
 \end{tabular}}
 \end{center}  \medskip

\noindent This construction is reviewed here. For $\CC,\CC'$ two Hurwitz algebras and $J=\mathcal{H}_3(\CC',\gamma)$, consider the vector space
\begin{equation}\label{eq_TitsModel}
\T(\CC,J)=\der\,(\CC)\oplus (\CC_0 \otimes J_0) \oplus \der\,(J),
\end{equation}
which is made into a Lie algebra over $\F$ by defining the multiplication $[\ ,\ ]$ on $\T(\CC,J)$ 
(bilinear and anticommutative)
which agrees with the ordinary commutator in $\der\,(\CC)$ and $\der(J)$ and it satisfies
\begin{equation}\label{eq_TitsProduct}
\begin{array}{l}
\bullet\  {[}\der(\CC),   \der(J)]=0, \\
\bullet\  {[}d, a\otimes x]=d(a) \otimes x, \\
\bullet\  [D, a\otimes x]=a \otimes D(x), \\
\bullet\  {[}a\otimes x, b\otimes y]= t_{J}(xy) d_{a,b}+[a,b]\otimes (x\ast y)+ 2t_\CC(ab)[R_x,R_y],
\end{array}
\end{equation}
for all $d\in \der(\CC)$, $D\in \der(J)$, $a,b\in \CC_0$ and $x,y\in J_0$. 

 In particular, note that $\T(\CC,\mathcal{H}_3(\CC',\gamma))$ and $\T(\CC',\mathcal{H}_3(\CC,\gamma))$  are both Lie algebras of type $E_6$
 if $\CC$ and $\CC'$ are Hurwitz algebras of dimensions $2$ and $8$, respectively.

\subsection{ $\sig{26}$ from Tits' construction}\label{subsec_construyendo la nuestra}

Several constructions of the real form $\sig{26}$ have appeared in the literature. 
According to the Vinberg's construction (\cite{Vinberg}, see also \cite[p.~178]{enci}), based also in two composition algebras $\CC$ and $\CC'$, the algebra 
$\der(\CC)\oplus\der(\CC')\oplus\mathfrak{sa}_3(\CC\otimes\CC')$ is isomorphic to $\sig{26}$ when $\CC=\OO$ and $\CC'=\R\oplus\R$. (Here $\mathfrak{sa}_3(\CC\otimes\CC')$ denotes the space of skew-hermitian matrices of order 3 with zero trace and entries in $\CC\otimes\CC'$.) Vinberg's approach to get Freudenthal's magic square has been used by \cite{otrocuadradoreales} and by some recent papers in the search of a unified description of the exceptional groups (see also \cite{Baez}). Unfortunately, this description of $\sig{26}$ does not suit very well with   our description of its gradings.  
Another approach due to Elduque   \cite{algoreales} gives  a construction based in symmetric composition algebras as $\sig{26}\cong\mathfrak{g}(p\OO,p(\R\oplus \R))$ and also by replacing the paraoctonion algebra $p\OO$ with the Okubo algebra \cite{libroOkubo}.   (This viewpoint has been useful for describing gradings over algebraically closed fields, but not over $\R$.) We devote this paragraph to show models   adapted to our   description of gradings in   Section~\ref{sec_gradings}.\smallskip

First, note that $\T( \R\oplus \R,J)$ is naturally isomorphic to $ \der(J)\oplus J_0$, which is $\Z_2$-graded with even and odd part $\der(J)$ and $J_0$ respectively, where the product is given by the natural action of $\der(J)$ on $J_0$ and $[x,y]:=[R_x,R_y]\in\der(J)$ for any $x,y\in J_0$. 

\begin{pr}\label{pr_constjacobson} (\cite{Jacobsondeexcepcionales})
$\T(\R\oplus \R,\mathcal{H}_3(\OO,\gamma))\cong\sig{26}$  for $\gamma=\diag\{1,-1,1\}$ and $\gamma=I$.
\end{pr}

\begin{proof}
 In fact, Jacobson described in   \cite[Eq.~(147)]{Jacobsondeexcepcionales}   all the real forms of $\e6$
 obtained when applying $\T(\CC,J)$   
 to a composition algebra $\CC$ of dimension 2: 
\begin{center}
 {\small
 \begin{tabular}{c||ccc|}
 $ \CC\,/\,J$ & $\mathcal{H}_3(\OO,I)$& $\mathcal{H}_3(\OO,\diag\{1,-1,1\})$& $\mathcal{H}_3(\OO_s,I)$\\
\hline \hline\vrule width 0pt height 12pt
 $\C$& $\sig{78}$&$\sig{14}$&$\mathfrak{e}_{6,2}$\\
 $ \R\oplus \R$& $\sig{26}$&$\sig{26}$&$\mathfrak{e}_{6,6}$\\ \hline
 \end{tabular}}
 \end{center}
 \end{proof}
 
 Second, if we now want to get real forms of $\e6$ using composition algebras $\CC$ of dimension 8 in  $\T(\CC,J)$, note that the Jordan algebra $\mathcal{H}_3( \R\oplus \R,I)$ (here $\overline{(a,b)}=(b,a)$ is the exchange involution) is naturally isomorphic to $\M:=\Mat_{3\times3}(\R)^+$.
 
 \begin{pr}\label{pr_constTitsnuestra}
$\T(\OO,\M)\cong\sig{26}$, being
$\M=\Mat_{3\times3}(\R)^+$.
\end{pr}

 \begin{proof}
 We are computing directly the signature of the Killing form. Observe the following facts:
 \begin{itemize}
 \item[a)]  $\der\,(\OO)$, $\OO_0 \otimes \M_0$ and $\der\,(\M)$ are three orthogonal subspaces for the Killing form.
 \item[b)] If $d,d'\in \der\,(\OO)$, then $k(d,d')=12\tr(dd')=3k_{\g2}(d,d')$, denoting by $k_{\g2}$ the Killing form of the algebra $\der(\OO)=\frak{g}_{2,-14}$.
 This implies that the signature of $k\vert_{\der\,(\OO)\times\der\,(\OO)}$ is the same as the one of $k_{\g2}$, that is, -14.
 \item[c)] If $D,D'\in \der\,(\M)$, then $k(D,D')=8\tr(DD')=8k_{\frak{a}_2}(D,D')$, denoting by $k_{\frak{a}_2}$ the Killing form of the algebra $ \der(\M)\cong\mathfrak{sl}(3,\R)=\frak{a}_{2,2}$.
 This implies that the signature of $k\vert_{\der\,(\M)\times\der\,(\M)}$ is the same   as the one of $k_{\frak{a}_2}$, that is, 2.
 \item[d)] For each $a,b\in\OO_0$ and $x,y\in\M_0$, the Killing form $k(a\otimes x,b\otimes y)=-60 n(a,b)t_\M(x\cdot y)$, for $n$ the polar form of the norm $n$ of the octonion algebra, that is, $n(a,b)=\frac12t_{\OO}(a\bar b)$. As $n$  is positive definite, the signature of $k\vert_{\OO_0 \otimes \M_0}$  will coincide with 7 times the signature of the traceform of $\M_0$ (the bilinear form $(x,y)\mapsto\tr(x\cdot y)$), equal to 2.  
 \end{itemize}
  Consequently, the signature of $\T(\OO,\M)$ turns out to be
 $-14+2-7\cdot 2=-26,$ as required.
 
 The proof of the above items can be done by following the next lines. 
 \begin{itemize} 
 
 \item For instance, for a), take   $d\in\der\,(\OO)$ and $D\in \der\,(\M)$, and note that $(\ad d\ad D)\vert_{\der\,(\OO)\oplus \der\,(\M)}=0$. Thus $k(d,D)=\tr (\ad d\ad D)\vert_{\OO_0 \otimes \M_0}= \tr (d\vert_{\OO_0}\otimes D\vert_{\M_0})=(\tr d\vert_{\OO_0})(\tr D\vert_{\M_0})=0$, since the endomorphisms contained in $\der(\OO)=[\der(\OO),\der(\OO)]$ have zero trace. 
 
 \item  For b), the relation $\tr(dd')=4k_{\g2}(d,d')$ is proved in \cite{Jacobsondeexcepcionales}. But $k(d,d')=k_{\g2}(d,d')+0+\tr(dd'\vert_{\OO_0}\otimes \id_{\M_0})$ and now we use again that the trace of the Kronecker product is the product of the traces. 
 
 \item  Similar arguments are applied in c) for $K(D,D')=\tr (\ad D\ad D')\vert_{\der\,(\M)}+\tr(\id_{\OO_0}\otimes DD'\vert_{\M_0})=k_{\frak{a}_2}(D,D')+7\tr(DD')$. Now the matrix of $DD'$ in a basis $\{v_i\}_{i=1}^8$ of $\M_0$ coincides with the matrix of $\ad D\ad D'$ in $\{\ad v_i\}_{i=1}^8$, which is a basis of $\der(\M)=\{[v,-]\mid v\in\mathfrak{sl}(3,\R) \}$. 
 
 \item For d) note that $\OO_0 \otimes \M_0$ is a $\der\,(\OO)\oplus \der\,(\M)$-irreducible module and the restriction $k\vert_{(\OO_0 \otimes \M_0)\times(\OO_0 \otimes \M_0)}$
 can be considered as an element in $\hom_{\der\,(\OO)\oplus \der\,(\M)}(S^2(\OO_0 \otimes \M_0),\R)$, which is a one-dimensional vector space.
 In consequence, there must exist $\alpha\in\R$ such that $k(a\otimes x,b\otimes y)=\a n(a,b)t_\M(x\cdot y)$  for any $a,b\in\OO_0$ and $x,y\in\M_0$.   Checking that this number $\a$ is negative (namely, $-60$) is a tedious task.   
 

 \end{itemize}

\end{proof}

\subsection{Model based on $\sp31  $}

    Take  $\L=\der(J)\oplus J_0$ for the Albert algebra $J=\mathcal{H}_3(\OO,I)$. 
    As in    Proposition \ref{pr_constjacobson},
     $\L\cong \T(\R\oplus \R,J)$ is a real form of $\e6$ of signature $-26$. 
    Let $\theta\in\aut(\L)$ be the automorphism given by $\theta\vert_{\der(J)}=\id$ and $\theta\vert_{J_0}=-\id$, in other words, the automorphism producing the $\Z_2$-grading on $\L$. In fact, $\theta$ is the Cartan involution related to $\L$ since $\der(J)\oplus\ii J_0$ is obviously compact. ($\mathcal{H}_3(\OO,I)$ is the Albert algebra, $t_J$ is positive definite and $\der(J)$ is the compact Lie algebra $\mathfrak{f}_{4,-52}$.)
    Let $\nu\colon J\to J$ be the automorphism of the Jordan algebra  fixing  $ {\mathcal{H}_3(\H,I)}$ and acting with eigenvalue $-1$ in the elements
    $$
    \left(\begin{array}{ccc}
    0&a&\bar b\\
    \bar a&0&c\\
    b&\bar c&0
    \end{array}\right),  
    $$
    for all $a,b,c\in\H^\perp=\H \textbf{l}$. Obviously $\nu$ is an order 2 automorphism such that $\dim\fix(\nu)=15$ and $\dim \{x\in J\mid \nu(x)=-x\}=12$.
    We denote by the same symbol $\nu$ to the automorphism of $\L$ given by
    $$
    d+x\mapsto \nu d\nu^{-1}+\nu(x)
    $$
    if $d\in\der(J)$ and $x\in J_0$. 
   As $\theta$ and $\nu $ commute, the automorphism $\nu':=\theta\nu=\nu\theta$ has again order two and  produces another $\Z_2$-grading
   $\L=\L_{\bar0}\oplus\L_{\bar1}$.
    The  subalgebra $\L_{\bar0} $ fixed by $\nu'$,  
   $$
   \L_{\bar0}= \{d\in\der(J)\mid \nu(d)=d\}\oplus \{x\in J_0\mid \nu(x)=-x\},
   $$ 
  has dimension $24+12=36$, since the first summand has type $C_3+A_1$ (the algebra $\der(\mathcal{H}_3(\H,I))$ is of type $C_3$).
     In particular  $\L_{\bar0}$ is a real form of   $\c4$ (see  Table \ref{tabladeautomorfysignaturas}). 
     Let us compute its signature with the help of Equation~(\ref{eq_signapartirparte fija}).  If $\sigma_0$ denotes the conjugation of
     $S_0=\L_{\bar0}^\C$ related to $\L_{\bar0}$ ($=S_0^{\sigma_0}$), let us check first that $\theta\sigma_0$ is a compact conjugation of $S_0$.
     For that aim, note that 
     $$
     S_0^{\theta\sigma_0}=\{d\in\der(J)\mid \nu(d)=d\}\oplus\ii \{x\in J_0\mid \nu(x)=-x\},
     $$
     is the even part of the $\Z_2$-grading of $\der(J)\oplus\ii J_0\cong\sig{78}$ produced by $\nu'$ (extended to the complexification $S=\L^\C$, and later restricted here), what implies its compactness.  
     Indeed, if $K=K_0\oplus K_1$ is a $\Z_2$-graded compact algebra with $K_0$ semisimple, then $K_0$ is compact too
     (one can choose a Cartan subalgebra of $K$ containing a Cartan subalgebra $\mathfrak h_0$ of $K_0$, thus any of the elements in $\mathfrak h_0$ has spectrum contained in $\R\ii$). 
     Now the compact conjugation $\theta\si_0$ allows us to compute the order two automorphism related to $\si_0$, that is,
         $\Phi([\si_0])=[\theta\vert_{\L_{\bar0}}]$. Thus
     $$
     \begin{array}{rl}
     \sign(\L_{\bar0})&=36-2\dim\fix(\theta\vert_{\L_{\bar0}})=36-2\dim(\fix(\theta)\cap\fix(\nu'))\\
    &=36-2\dim(\fix(\theta)\cap\fix(\nu))=36-2\cdot 24=-12,
     \end{array}
     $$
     which forces $\L_{\bar0}$ to be isomorphic to $\sp31=\{x\in\Mat_{4\times4}(\H)\mid x^tI_{31}+I_{31}\bar x=0\}$ ($I_{31}=\diag\{1,1,1,-1\}$), the only real form of $\c4$ with signature $-12$. To summarize, we have proved the following result:
     
     \begin{pr}
     There exists an $\sp31$-irreducible module $\mathcal{U}$ such that $\sig{26}=\sp31\oplus \mathcal{U}$.
     \end{pr}
     
     We will also need to know that the existence of a subalgebra isomorphic to $\sp31$ determines, in some sense, the real form. For $L=L_0\oplus V$   a real $\Z_2$-graded algebra and $t\in\R$, denote by $L^t:=(L,[\ ,\ ]^t)$ the Lie algebra with the same underlying vector space but new product given by
     \begin{equation}\label{eq_Lapareada}
     [x_0+v_0,x_1+v_1]^t=[x_0,x_1]+[x_0,v_1]+[v_0,x_1]+t[v_0,v_1],
     \end{equation}
   if $x_i\in L_0$ and $v_i\in V$. If $t>0$, the map   which sends $x_0+v_0$ to $x_0+\sqrt{t}v_0$ is a Lie algebra isomorphism between $L$ and $L^t$, but in general $L$ is not isomorphic to $L^{-1}$.  

     \begin{pr}\label{pr_lasquecontienensp31}
     If there is an $\sp31$-irreducible module  $\mathcal{U}'$ such that $\sp31\oplus \mathcal{U}'$ is a $\Z_2$-graded real form of $\e6$, then this real form must have signature either $-26$ or $2$.
     \end{pr}
     
     \begin{proof}
     Denote   $\L=\sp31\oplus \mathcal{U}$ and   $\L'=\sp31\oplus \mathcal{U}'$. After complexifying, we get in both cases a decomposition of $\e6$ as a sum of $\c4$ and a $\c4$-module. But this decomposition is unique, so that the $\c4$-modules $\mathcal{U}^\C$ and $\mathcal{U}'^\C$ are necessarily isomorphic, of type $V(\lambda_4)$ for $\lambda_4$ the fundamental weight (\cite{Humphreysalg}).
     As $\mathcal{U}^\C=\mathcal{U}\oplus\ii\mathcal{U}$ is a real $\sp31$-module isomorphic to $\mathcal{U}\oplus \mathcal{U}$, then $ \mathcal{U}$ is a submodule of $ \mathcal{U}\oplus  \mathcal{U}\cong \mathcal{U}'\oplus \mathcal{U}'$ and hence one of the projections   $\pi_i\colon  \mathcal{U}\to \mathcal{U}'$ on each of the two copies of $ \mathcal{U}'$ will be nonzero ($i=1,2$). By irreducibility, such $\pi_i$ will be an isomorphism of $\sp31$-modules which will allow to identify $ \mathcal{U}$  and $ \mathcal{U}'$. Thus   we can replace $ \mathcal{U}'$ with $ \mathcal{U}$ in the definition of $\L'$ without loss of generality, so that now $\L$ and $\L'$ coincide as vector spaces, as well as the bracket on the even par $\sp31$ and the action of the even part on the odd one. Denote by $[\ ,\ ]$ (respectively $[\ ,\ ]'$) the restriction of the bracket $\mathcal{U}\times\mathcal{U}\to \L$ (respectively $\mathcal{U}\times\mathcal{U}\to \L'$) to $\mathcal{U}\times\mathcal{U}\to \sp31$. Thus we can consider both 
$[\ ,\ ],[\ ,\ ]'\in\hom_{\sp31}(\Lambda^2(\mathcal{U}),\sp31)$. This space has real dimension equal to 1, since 
$\dim_{\C}\hom_{\c4}(\Lambda^2(V(\lambda_4)),\c4\cong V(\lambda_2))=1$ (a well-known fact of representation theory). Hence there is $t\in\R$ such that $[u,v]'=t[u,v]$ for all $u,v\in\mathcal{U}$. 
Thus $\L'\cong\L^t$ (Equation~(\ref{eq_Lapareada})), so that $\L'$ is isomorphic either to $\L$ or to $\L^{-1}$, according to $t$ being either positive or negative.

Let us check that the signature of the Killing form $k^{-1}$ of $\L^{-1}$ is equal to 2 by relating it with the signature of the Killing form $k$ of $\L$. First note that $\sp31$ and   $\mathcal{U}$  
  are orthogonal relative to both $k$ and $k^{-1}$, since we have a $\Z_2$-grading $\L=\L_{\bar0}\oplus \L_{\bar1}$, so, if $x\in\L_{\bar0}$ and $v\in\L_{\bar1}$, then $\ad^{\pm1} x\ad^{\pm1} v$ interchanges $ \L_{\bar0}$ and $ \L_{\bar1}$ and is a zero trace endomorphism. Thus, the related signatures satisfy
 \begin{equation}\label{eq_signLp}
 \begin{array}{l}
 \sign(k)=\sign(k\vert_{\L_{\bar0}\times\L_{\bar0}})+
 \sign(k\vert_{\mathcal{U}\times\mathcal{U}}),\\
 \sign(k^{-1})=\sign(k^{-1}\vert_{\L_{\bar0}\times\L_{\bar0}})+
 \sign(k^{-1}\vert_{\mathcal{U}\times\mathcal{U}}).
 \end{array}
 \end{equation}
 But for any $t\in\R$ we have $\ad^{t}x=\ad x$ if $x\in\L_{\bar0}$ and $\ad^{t}u\ad^{t}v=t\ad u\ad v$ if $u,v\in\mathcal{U}$, so that $k^{-1}(x,y)=k(x,y)$ and $k^{-1}(u,v)=-k(u,v)$ ($y\in\L_{\bar0}$). Thus $\sign(k\vert_{\L_{\bar i}\times\L_{\bar i}})=
 (-1)^i\sign(k^{-1}\vert_{\L_{\bar i}\times\L_{\bar i}})$, and by Equation~(\ref{eq_signLp}) we get that 
 $$ 
 \sign(k)+ \sign(k^{-1})=2\sign(k\vert_{\L_{\bar0}\times\L_{\bar0}}).
 $$
Denote by $k_0$ the Killing form of $\L_{\bar0}=\sp31$. Both $k_0$ and 
$k\vert_{\L_{\bar0}\times\L_{\bar0}}$ are symmetric $\L_{\bar0}$-invariant bilinear forms, so that they can be considered as elements in 
$\hom_{\sp31}(S^2(\L_{\bar0}),\R)$, which is a one-dimensional vector space since its complexification is $\hom_{\c4}(S^2(V(\lambda_2)),\C\cong V(0))$  and the decomposition of $S^2(V(\lambda_2))$ into sum of irreducible submodules is $V(4\lambda_1)\oplus V(2\lambda_2 )\oplus V(\lambda_ 2)\oplus V(0)$. This means that there is $\delta\in\R$ such that $k\vert_{\L_{\bar0}\times\L_{\bar0}}=\delta k_0$. With some work in the complexified algebra, we can prove that in fact $\delta=\frac{12}{5}$. Consequently  $\sign(k\vert_{\L_{\bar0}\times\L_{\bar0}})$ coincides with $\sign k_0=-12$, so that $\sign k^{-1}=2(-12)-(-26)=2$, which finishes the proof.
         \end{proof}


\section{Preliminaries on gradings }

\subsection{Basic concepts }

The main reference about the topic is \cite{libro} (consult \cite{LGI} too for background on gradings on Lie algebras).

 Let $\A$ be a finite-dimensional algebra  over  $\F$, and $G$ an abelian group.
  A \emph{$G$-grading} $\Gamma$ on $\A$ is a vector space decomposition
$
 \Gamma: \A = \bigoplus_{g\in G} \A_g
 $
such that
$
 \A_g \A_h\subset \A_{g+h}$ for all $g,h\in G$.
 The subspace $\A_g$ will be referred to as \emph{homogeneous component of degree $g$} and its nonzero elements will be called   \emph{homogeneous elements of degree $g$}. The \emph{support} of the grading is the set $\supp \Gamma :=\{g\in G\mid \A_g\neq 0\}$.
 We will assume from now on that $\supp \Gamma$ generates $G$ (there is  no loss of generality).  The \emph{type} of   $\Gamma$ is the sequence of numbers $(h_1,\ldots, h_r)$ where $h_i$ is the number of homogeneous components of dimension $i$, with $i=1,\ldots, r$ and $h_r\neq 0$. Obviously, $\dim \mathcal{A} = \sum_{i=1}^r ih_i$.

  If $\Gamma\colon \A=\oplus_{g\in G} \A_g$ and $\Gamma'\colon \A=\oplus_{h\in H} \A'_{h}$ are gradings over two abelian groups $G$ and $H$, $\Gamma$ is said to be a \emph{refinement} of $\Gamma'$ (or $\Gamma'$ a \emph{coarsening} of $\Gamma$) if for any  $g\in G$ there is $h\in H$ such that $\A_g\subset \A'_{h}$.  A refinement is \emph{proper} if some inclusion $\A_g\subset \A'_{h}$ is proper. A grading is said to be \emph{fine} if it admits no proper refinement.
 Also, $\Gamma$ and $\Gamma'$ are said to be \emph{equivalent} if there is an algebra isomorphism  $\varphi\colon\A \rightarrow \A$ and a bijection $\alpha\colon  \supp \Gamma \rightarrow \supp \Gamma'$ such that $\varphi(\A_s)=\A'_{\alpha(s)}$ for all $s\in \supp \Gamma$. 
We are interested in classifying fine gradings up to equivalence, because any grading is obtained as a coarsening of some fine one.

For any $G$-grading $\Gamma\colon \A=\oplus_{g\in G} \A_g$,
the group $U(\Gamma)$ generated by $ \supp \Gamma$ with defining relations $s_1s_2 = s_3$ whenever $0\ne \A_{s_1}\A_{s_2}\subset \A_{s_3}$ is called \emph{universal group of} $\Gamma$, because it satisfies the following universal property: there is   a grading on $\A$ over $U(\Gamma)$ equivalent to $\Gamma$ such that for any other
       $G'$-grading $\Gamma'$ on $\A$ equivalent to $\Gamma$, there exists a unique group homomorphism $U(\Gamma)\to G'$ that restricts to the identity on $\supp \Gamma$.  
  Observe that the group $U(\Gamma)$ is necessarily abelian in case that $\A$ is a  simple Lie algebra \cite[Proposition 1.12]{libro}.

In the complex case, there is a duality between (abelian group) gradings and actions, which has been very useful for studying fine gradings.
If $\A=\oplus_{g\in G} \A_g$ is a $G$-grading, the map $\psi\colon  \X(G):=\text{hom} (G,\C^\times) \rightarrow
\aut(\A)$  which sends each character $\alpha \in \X(G)$ to the automorphism $\psi_{\alpha}\colon \A \rightarrow \A$ given by $\A_g\ni x \mapsto \psi_{\alpha}(x) :=\alpha (g) x$ is a group homomorphism. Since $G$ is finitely generated ($\A$ has finite dimension and $G$ is generated by the support), $\psi(\X(G))$ is an algebraic quasitorus. Conversely, if $Q$ is a quasitorus and $\psi\colon  Q \rightarrow \aut (\A)$ is a homomorphism, $\psi(Q)$ consists of semisimple automorphisms and we have a $\X(Q)$-grading $\A=\oplus_{g\in\X(Q)} \A_g$ given by $\A_g=\{x\in \A\mid \psi(q)(x)=g(q)x \  \forall q \in Q\}$, with $\X(Q)$ a finitely generated abelian group. Furthermore, if the original grading is fine, the quasitorus $\psi(\X(G))\le\aut(\A)$ is maximal (for $G$ the universal group of the grading) and conversely. To be precise, there is a one-to-one correspondence between   equivalence classes of fine gradings on (the complex algebra) $\A$ and   conjugacy classes of maximal quasitori of the group $\aut(\A)$ \cite[Proposition~1.32]{libro}. These maximal quasitori are also called \emph{MAD-groups}  (Maximal Abelian Diagonalizable). Unfortunately, the knowledge of the MAD-groups of $\aut(\L)$ for $\L$ a real Lie algebra, is not an equivalent problem to that one of classifying fine gradings on $\L$ up to equivalence \cite[\S4]{review}.

\subsection{Fine gradings on $\e6$ }

The fine gradings on the exceptional complex simple Lie algebra $\e6$ were classified (up to equivalence) in \cite{e6}, although also \cite{libro} and \cite{proceedingsLecce} contain  alternative descriptions of all these fine gradings. According to the classification, there are 14 fine gradings on $\e6$, whose universal groups and types are exhibited in   Table \ref{tablafinasdee6}.
Their symmetry groups are computed in \cite{Weyle6}, which contains several models adapted to the various gradings.

\smallskip

\begin{table}
\begin{tabular}{|c|c|c| }
\hline\vrule width 0pt height 10pt
 Grading & Universal group & Type   \cr
\hline\hline\vrule width 0pt height 12pt
 $\Gamma_{1}$ & $ \Z_3^4$   &  $( 72,0,2 )$   \cr
  \hline \vrule width 0pt height 13 pt
 $\Gamma_{2}$ & $\Z^2\times\Z_3^2$   &  $( 60,9 )$   \cr
\hline  \vrule width 0pt height 13 pt
$\Gamma_{3}$ & $\Z_3^2\times\Z_2^3$   &  $( 64,7 )$  \cr
\hline \vrule width 0pt height 13 pt
$\Gamma_{4}$ & $\Z^2\times\Z_2^3$   &  $( 48,1,0,7 )$ \cr
\hline \vrule width 0pt height 13 pt
$\Gamma_{5}$ & $\Z^6 $   &  $(72,0,0,0,0,1  )$  \cr
\hline \vrule width 0pt height 13 pt
$\Gamma_{6}$ & $\Z^4\times\Z_2$   &  $(72,1,0,1  )$  \cr
\hline \vrule width 0pt height 13 pt
$\Gamma_{7}$ & $ \Z_2^6$   &  $( 48,1,0,7 )$  \cr   
\hline \vrule width 0pt height 13 pt
$\Gamma_{8}$ & $\Z\times\Z_2^4$   &  $(  57,0,7)$  \cr
\hline \vrule width 0pt height 13 pt
$\Gamma_{9}$ & $\Z_3^3\times\Z_2$   &  $( 26,26 )$  \cr
\hline \vrule width 0pt height 13 pt
$\Gamma_{10}$ & $\Z^2\times\Z_2^3$   &  $(60,7,0,1   )$  \cr
\hline \vrule width 0pt height 13 pt
$\Gamma_{11}$ & $\Z_4\times\Z_2^4$   &  $(48,13,0,1     )$ \cr
\hline \vrule width 0pt height 13 pt
$\Gamma_{12}$ & $\Z \times\Z_2^5$   &  $( 73,0,0,0,1  )$ \cr
\hline \vrule width 0pt height 13 pt
$\Gamma_{13}$ & $ \Z_2^7$   &  $( 72,0,0,0,0,1  )$ \cr
\hline \vrule width 0pt height 13 pt
$\Gamma_{14}$ & $ \Z_4^3$   &  $( 48,15 )$ \cr
\hline %
 \end{tabular}\vspace{4pt} %
 
 \caption{Fine gradings on $\e6$}\label{tablafinasdee6}
 \end{table}
 
\subsection{Gradings on real forms }\label{sec_gradreales}

  If $\Gamma: L = \oplus_{g\in G} L_g$ is a grading on a real Lie algebra $L$, let us denote by $\Gamma^\C$ the grading on $L^\C$ given by   $\Gamma^{\C}: L^{\C}=  \oplus_{g\in G} (L_g)^\C=\oplus_{g\in G} (L_g \oplus \ii L_g)$. %
  These \emph{complexified gradings} will be very useful in the study of gradings on real Lie algebras.

 \begin{defi}
Let $\Gamma_1: S = \oplus_{g\in G} S_g$ be a grading on a complex Lie algebra $S$ and let $L$ be a real form of $S=L^\C$. We will say that $L$ \emph{inherits} the grading $\Gamma_1$ if there exists a grading  $\Gamma$ on $L$ such that $\Gamma^\C=\Gamma_1$. This happens if and only if $L$ is a graded subspace of $S$; that is, $L=\oplus_{g\in G} (L\cap S_g)$.
\end{defi}

\begin{re}{\rm
By abuse of notation, we will say that $\sig{26}$ inherits the grading $\Gamma_i$ in Table \ref{tablafinasdee6} ($i=1,\dots,14$) if some real form of $\e6$ of signature $-26$ has a grading whose complexified grading is \emph{equivalent} to $\Gamma_i$ (hence any other real form of $\e6$ of signature $-26$ will have it). Take into account that in Table \ref{tablafinasdee6} the notation  $\Gamma_i$ means certain equivalence class of fine gradings on  $\e6$, in which   a determined representative has not been  previously fixed. }
\end{re}

 As we are interested in fine gradings, note the following obvious result.
\begin{pr}\label{pr_inheritedFineGrading}
Let $\Gamma$ be a grading on  a real form  $L$ of a complex Lie algebra $S$. If $\Gamma^{\C}$ is a fine grading on $S$, then $\Gamma$ is a fine grading on $L$.
\end{pr}
Unfortunately, the converse result does not seem to be true: there could exist a fine  grading $\Gamma$ on $L$  such that $\Gamma^{\C}$ is not fine, although, as far as we know, none example has appeared.

Some different approaches have been used to study gradings on real forms. Most of them appear in the review paper \cite{review}. For instance, the so-called \emph{fundamental method}, used in \cite[Theorem~4]{review}, can be stated as follows: 

\begin{pr}\label{pr_metodobasereal}
Let $\Gamma: S = \oplus_{g\in G} S_g$ be a grading on a complex Lie algebra $S$ and let $\si$ be a conjugation. If there is a basis of $S$ formed by homogeneous elements, all of them $\si$-invariant, then the real form $S^\si$ inherits the grading $\Gamma$.
\end{pr}

 In general,  finding this basis is a problem as difficult as the original one (in fact, the converse of Proposition~\ref{pr_metodobasereal} is also true, if a real form $L$ of $S$ inherits $\Gamma$, then there exists such a basis). A new viewpoint appears in \cite[Proposition~3]{reales}:
 
 \begin{pr}\label{pr_nuestrometodo}
 Let $\si$ be a conjugation of $S$ such that $S^\si$ inherits certain fine grading $\Gamma: S = \oplus_{g\in G} S_g$.
 For any other conjugation $\mu$ of $S$, the real form $S^\mu$ inherits $\Gamma$ if and only if $\mu\si^{-1}$ belongs to the (maximal) quasitorus of $\aut(S)$ producing the grading. 
 \end{pr}
This   is used  in \cite{reales} in order to give a complete list up to equivalence of the fine gradings on the real forms of $\g2$ and $\f4$: for any $Q$ maximal quasitorus of $\aut(S)$, being $S\in\{\g2,\f4\}$, the problem consists of classifying the signatures of the Killing forms of $S^{\si q}$ for all those $q\in Q$ with $(\si q)^2=\id.$
This problem is equivalent to that one of computing the dimensions  of the subalgebras  fixed by some automorphism in the class  $\Phi([\si q])$ (see Equation~(\ref{eq_defdeFI})), which is not usually an easy task
 since one first has to find the 
automorphisms $\theta_{\si q}$.  Both methods of propositions \ref{pr_metodobasereal} and \ref{pr_nuestrometodo} will be used through  this paper.


\section{Fine gradings on $\sig{26}$ }\label{sec_gradings}

\subsection{A $ \Z_2^6$-grading and a $\Z \times\Z_2^4$-grading.}

 Recall from Proposition~\ref{pr_constjacobson}   that $\L=\der\J\oplus \J_0$ has signature $-26$, being $\J=\mathcal{H}_3(\OO,\diag\{1,-1,1\})$. All the $G$-gradings on $\J$ induce naturally $G$-gradings on $ \der(\J)$ and $G\times\Z_2$-gradings on $\L$. The Jordan algebra $\J$ is $\Z_2^5$-graded and $\Z\times\Z_2^3$-graded \cite[Corollary~1.3]{reales} (there, $\J$ is the nonsplit and noncompact real form   of the complex Albert algebra).
 Furthermore, the $ \Z_2^6$-grading and the $\Z \times\Z_2^4$-grading induced on $\L$ have as complexifications $\Gamma_7$ and $\Gamma_8$ respectively.  In particular, the real form $\sig{26}$ admits a fine $ \Z_2^6$-grading and a fine $\Z \times\Z_2^4$-grading.
 
 For completeness we next recall descriptions of these gradings on $\J$. Denote a generic element in $\J$ by
 $$
 \left(\begin{array}{ccc}
 s_1&a_3&\bar a_2\\-\bar a_3&s_2&a_1\\a_2&-\bar a_1&s_3\end{array}\right)=:\sum_i s_iE_i+\sum_i\iota_i(a_i)
 $$
 if $s_i\in\R$ and $a_i\in\OO$.
 On one hand, the octonion algebra $\OO$ admits a $\Z_2^3$-grading:
 \begin{equation}\label{eq_gradenO}
 \begin{array}{lll}
 \OO_{(\bar0,\bar0,\bar0)}=\R 1, & \quad& \OO_{(\bar0,\bar0,\bar1)}=\R\textbf{l},\\
 \OO_{(\bar1,\bar0,\bar0)}=\R \textbf{i}, &&  \OO_{(\bar1,\bar0,\bar1)}=\R\textbf{il},\\
 \OO_{(\bar0,\bar1,\bar0)}=\R \textbf{j}, &&  \OO_{(\bar0,\bar1,\bar1)}=\R\textbf{jl},\\
 \OO_{(\bar1,\bar1,\bar0)}=\R \textbf{k}, &&  \OO_{(\bar1,\bar1,\bar1)}=\R\textbf{kl},
 \end{array}
\end{equation}
which induces the $\Z_2^3$-grading on the Jordan algebra $\J$ given by
\begin{equation}\label{eq_gradingdeoctoniones}
\begin{array}{l}
\J_{(\bar0,\bar0,\bar0)}=\sum_i\R E_i\oplus\sum_i\R\iota_i(1),\\
\J_g=\sum_i\iota_i(\OO_g) \quad \textrm{ if }e\ne g\in\Z_2^3.
\end{array}
\end{equation}
On the other hand, $\J$ is $\Z_2^2$-graded  as
$$
\begin{array}{ll}
\J_{(\bar0,\bar0)}=\sum_i\R E_i, \quad &\J_{(\bar0,\bar1)}=\iota_1(\OO),\\
\J_{(\bar1,\bar0)}=\iota_2(\OO),&\J_{(\bar1,\bar1)}=\iota_3(\OO),
\end{array}
$$
and also $\Z$-graded as
$$
\begin{array}{l}{\J }_{-2}=\R({ E_2-E_3-\iota_1(1)}),\\
\J_{-1}=\{{ \iota_2(x)-\iota_3(\overline{x})\mid
x\in \OO}\},\\
{\J}_{0}= \span{\{E_1, E_2+E_3,
\iota_1(x)\mid x\in \OO_0\}},\\
{\J}_{1}=\{{
\iota_2(x)+\iota_3(\overline{x})\mid x\in \OO}\},\\
{\J}_{2}=\R({  E_2-E_3+\iota_1(1)}),
\end{array}
$$ 
which is precisely the eigenspace decomposition relative to the endomorphism
$4[R_{\iota_1(1)},R_{E_2}]\in\der(\J)$,
which has integer eigenvalues.  
The last two gradings are compatible with that one in Equation~(\ref{eq_gradingdeoctoniones}), so that they can be combined to produce the desired gradings on $\J$.

\subsection{A $\Z^2\times\Z_2^3$-grading.}

Gradings on the  two ingredients involved in Tits' construction    can be used to get some interesting gradings on the resulting Lie  algebras (in fact, this was what we did in the above subsection, by taking the natural $\Z_2$-grading on $\R\oplus\R$).
Namely, if $\CC=\oplus_{g\in G}\CC_g$ and $J=\oplus_{h\in H}J_h$ are $G$ and $H$-gradings on the composition algebra $\CC$ and on the Jordan algebra $J$ respectively, then $\der(\CC)$ and $\der( J)$ are also $G$ and $H$-graded (respectively), and thus $L=\T(\CC,J)=\oplus_{(g,h)\in G\times H}L_{(g,h)}$ is $G\times H$-graded, for
\begin{equation}\label{eq_mixinggradings}
\begin{array}{l}
L_{(e,e)}= \der(\CC)_e \oplus\der (J)_e \oplus (\CC_0)_e\otimes (J_0)_e,\\
L_{(e,h)}= \der (J)_h \oplus (\CC_0)_e\otimes (J_0)_h,\\
L_{(g,e)}=\der(\CC)_g \oplus (\CC_0)_g\otimes (J_0)_e,\\
L_{(g,h)}=(\CC_0)_g\otimes (J_0)_h, 
\end{array}
\end{equation}
if $e\ne g\in G$, $e\ne h\in H$ (note that $\CC_0$ and $J_0$ are necessarily graded subspaces of $\CC$ and $J$, respectively).

Now we consider a $\Z^2$-grading on $\M=\Mat_{3\times 3}(\R)^+$: for $g_1=(0,0), g_2=(1,0)$ and $g_3=(0,1)$, we state that the degree of 
   the unit matrix $E_{ij}$ is $g_j-g_i$, being $E_{ij}$ the matrix with the only nonzero entry, the   $(i,j)$\emph{th}, equal to $1$.
    In fact, this provides a grading on the associative matrix algebra $ \Mat_{3\times 3}(\R)$, in particular also with the symmetrized product. Then we  combine, following Equation~(\ref{eq_mixinggradings}), this $\Z^2$-grading on $\M$ with the $\Z_2^3$-grading on $\OO$ given by Equation~(\ref{eq_gradenO}), thus getting a 
  $\Z^2\times\Z_2^3$-grading on $\T(\OO,\M)$  whose complexification is $\Gamma_4$. But $\T(\OO,\M)$ has signature $-26$ by Proposition~\ref {pr_constTitsnuestra}.

\subsection{A $\Z_4\times\Z_2^4$-grading.} 

We begin by reviewing a concrete description of the fine  $\Z_4\times\Z_2^4$-grading $\Gamma_{11}$ on the complex algebra $S=\e6$. 
Consider the  matrices
$$
\sigma_1= \left(\begin{array}{cc}
0&1\\
1&0
\end{array}\right),\qquad
 C=\left(\begin{array}{cc}
0 & I_4 \\
-I_4 & 0 
\end{array}\right).
$$
Then the invertible matrices
$$
\begin{array}{l}
A_1=\ii\left(\begin{array}{cccc}
0 &0 &  I_2 & 0\\
0 & 0 & 0 & \s_1 \\
 I_2 & 0 & 0 & 0\\
0 & \s_1 & 0 & 0
\end{array}\right) ,\\
A_2=\ii\diag\{{I_4,-I_4}\} ,\\
A_3=\diag\{{\sigma_1,\sigma_1,\sigma_1,\sigma_1}\},\\
A_4=\diag\{{1,-1,-\ii,\ii,1,-1,\ii,-\ii}\} ,
\end{array}
$$
satisfy $A_iC A_i^t=C$, so that $\Ad A_i\in\aut(\mathfrak{sp}_\C(8,C))\cong \Sp_\C(8,C)$, where the complex Lie algebra 
$ S_0=\mathfrak{sp}_\C(8,C)=\{x\in\Mat_{8\times8}(\C)\mid xC+Cx^t=0\}$ is of type $\c4$. The group generated by $\{A_i\}_{i=1}^4$ is a MAD-group of $\Sp_\C(8,C)$ which produces a simultaneous diagonalization of $S_0$ of type $(24,6)$. It is not difficult to check (by doing the simultaneous diagonalization) that there is $B_0$ a basis of $S_0$ formed by simultaneous eigenvectors such that each of them is a matrix with entries in the set $\{1,0,-1\}$. As $B_0\subset  \mathfrak{sp}_\R(8,C)=\{x\in\Mat_{8\times8}(\R)\mid xC+Cx^t=0\}$, then the real vector space $L_0$ spanned by $B_0$ coincides with $\mathfrak{sp}_\R(8,C)$. In particular $L_0$ is a split real form of $\c4$ (Remark~\ref{rem_compactac4}) and deserves the notation $\mathfrak{c}_{4,4}$. Obviously $L_0$ inherits the $ \Z_4 \times\Z_2^3$-grading on $S_0$ by Proposition~\ref{pr_metodobasereal}.

Recall \cite{e6} that $S$ can be constructed from $S_0$ as $S=S_0\oplus\ker c$, 
being $c$ the contraction
$$
\begin{array}{llcl}
c\colon  &\Lambda^4\C^8& \longrightarrow &\Lambda^2\C^8\\
&v_1\wedge v_2\wedge v_3 \wedge v_4  &\longmapsto&\sum_{\tiny
{\begin{array}{l}\sigma\in S_4\\
\sigma(1)<\sigma(2)\\
\sigma(3)<\sigma(4)\end{array}}}
(-1)^\sigma  (v_{\sigma(1)}^tCv_{\sigma(2)})
v_{\sigma(3)}\wedge v_{\sigma(4)},
\end{array}
$$
which is of course a homomorphism of $S_0$-modules. 
Moreover, if $A\in\Sp_\C(8,C)$, the map $A^\bullet\colon S\to S$ given by $\Ad A$ in $S_0$ ($X\mapsto AXA^{-1}$) and sending $v_1\wedge v_2\wedge v_3 \wedge v_4$ to $Av_1\wedge Av_2\wedge Av_3 \wedge Av_4$ is (well-defined and) an automorphism of the complex algebra $S$. If $\theta$ denotes the outer order two automorphism of  $S=S_0\oplus\ker c$   providing the $\Z_2$-grading (with odd part $\ker c $), then   $\{A^\bullet\mid A\in\Sp_\C(8,C)\}\times\langle \theta\rangle$ coincides with the centralizer of $\theta$ in the group $\aut (S)$, so that the group generated by $ \{A_i^\bullet, \theta\mid i=1,\dots4\}$ is a    
finite MAD-group of $\aut (S)$, isomorphic as abstract group to
$\Z_4 \times\Z_2^4$. Again, when we realize the simultaneous diagonalization, there is a homogeneous basis $B_1$ of $\ker c$ formed by linear combinations of certain elements in $\{v_{i_1}\wedge v_{i_2}\wedge v_{i_3} \wedge v_{i_4} \mid 1\le i_1<i_2<i_3<i_4\le8\}$ (this set is a $\C$-basis of $\Lambda^4\C^8 $,  where  $v_i$ denotes the column vector of $\C^8$ with the only nonzero entry being the i\emph{th} one, equal to 1), all these linear combinations with coefficients $\pm1$. 
(This basis $B_1$ is exhibited in \cite{tesis}, where   long  computations by  hand  were made.) 
That is, $B_1$ lives in $\ker c\cap \Lambda^4\R^8 $, which  allows to assure that $L$, the real vector space spanned by $B_0\cup B_1$,
 not only satisfies $[L,L]\subset S$ but $[L,L]\subset L$, that is, $L$ is a real form of $S=\e6$ inheriting $\Gamma_{11}$.  
To summarize, we have applied Proposition~\ref{pr_metodobasereal} to prove   item a) in:

\begin{pr}\label{hereda11}
Let $\Gamma_{11}$ be the fine $\Z_2^4\times\Z_4$-grading on $\e6$ given as in \cite{e6}. 
\begin{itemize}
\item[a)]
There is   $L=L_0\oplus L_1$ a $\Z_2$-graded real form of  $\e6$ inheriting $\Gamma_{11}$ and such that $L_0\cong\mathfrak{c}_{4,4}$.
\item[b)]
There is   $L'=L_0'\oplus L_1'$ a $\Z_2$-graded real form of  $\e6$ inheriting $\Gamma_{11}$ and such that $L'_0\cong\mathfrak{c}_{4,-12}$.
\item[c)] The real form $\sig{26}$ inherits $\Gamma_{11}$.
\end{itemize}
\end{pr}

\begin{proof}
Let $\sigma$ be the conjugation of $S=\e6$ fixing $L$, and let $\sigma_0$ be the conjugation of $S_0=L_0^\C=\mathfrak{sp}_\C(8,C)$ fixing $L_0$. According to
Proposition~\ref{pr_nuestrometodo}, also $\sigma'=\sigma A_1^\bullet A_2^\bullet A_3^\bullet$ is a conjugation of $S$ whose related real form $L'=S^{\sigma'}$ inherits $\Gamma_{11}$. The algebra $L'$ is $\Z_2$-graded, being its even part $L'_0=\{x\in S_0\mid \sigma_0'(x)=x\}$ for the conjugation $\sigma_0'=\sigma_0\Ad (A_1  A_2  A_3)$.  
In order to compute the signature of the Killling form of $L'_0$, which is a real form of $\c4$, we apply the map $\Phi$  in Equation~(\ref{eq_defdeFI}). 
Note that $\sigma_0\Ad C$ is a compact conjugation of $S_0$  by Remark~\ref{rem_compactac4}. Besides, $\Ad C\in\Sp_\C(8,C)$ commutes with $\Ad A_i$  for all $i=1,2,3$, 
so that $\Phi([\sigma_0'])=[\Ad (C  A_1  A_2 A_3)]$. The subalgebra fixed by this order two automorphism has dimension 24 (an easy  matrix computation), so the signature of $k_{L_0'}$ is $36-2\cdot 24=-12$ by Equation~(\ref{eq_signapartirparte fija}), and item b) follows.

By Proposition~\ref{pr_lasquecontienensp31}, either this algebra $L'$   has signature $-26$ or $L'^{-1}$ has signature $-26$. But $L'^{-1}$ also inherits the grading $\Gamma_{11}$, since the homogeneous basis of $L'$  is also a homogeneous basis of $L'^t$ for any $t\in\R$. In this way, there is a real form of $\e6$ of signature $-26$ (and another one of signature 2) inheriting $\Gamma_{11}$.
\end{proof}

\begin{re}\label{rem_compactac4}
{\rm
Recall that $ \mathfrak{sp}_\R(8,C)=\{x\in\Mat_{8\times8}(\R)\mid xC+Cx^t=0\}$ is a split real form of $\c4$ since 
$\h=\sum_{i=1}^4\R h_i$ for $h_i=E_{ii}-E_{i+4,i+4}$ diagonalizes $ \mathfrak{sp}_\R(8,C)$ with real eigenvalues. The set $\{\alpha_i\colon\h\to\R\mid i=1,\dots4\}$ is a basis of the root system relative to $\h$, if $\alpha_1(h)=w_1-w_2$, $\alpha_2(h)=w_2-w_3$, $\alpha_3(h)=w_3-w_4$ and $\alpha_4(h)=2w_4$, for $h=\sum_i w_ih_i$.
A basis as $\mathcal B$   in Section~\ref{sec_prelimirealforms} can be obtained by taking
$$\begin{array}{ll}
e_{ 1}=E_{12}-E_{65}, \qquad&f_{1}=E_{21}-E_{56},  \\
e_{ 2}=E_{23}-E_{76}, &f_{2}=E_{32}-E_{67},  \\
e_{ 3}=E_{34}-E_{87}, &f_{3}=E_{43}-E_{78},  \\
e_{4}=E_{48}, &f_{4}=E_{84} .
\end{array}
$$
If we take into account that $\Ad( C )$ swaps $E_{i,j}$ with $E_{i+4,j+4} $ and $E_{i,j+4}$ with $-E_{i+4,j} $   for all $i,j=1,\dots,4$, then $\Ad( C )(e_j)=-f_j  $  and $\Ad( C )(f_j)=-e_j  $, so that $\sigma_0 \Ad( C )$ is compact ($\Ad( C )$ is the automorphism $\omega$ in Section~\ref{sec_prelimirealforms}).
}
\end{re}

\subsection{Not more fine gradings from $\e6$.}

Our aim now is to prove that there are not gradings on $\sig{26}$ whose complexified grading is equivalent to $\Gamma_i$ in Table~\ref{tablafinasdee6} for some $i\ne 4,7,8,11$.
Note   first   the following trivial linear algebra result (\cite[Lemma~1]{Heis}, particularized to the real field).

\begin{lm} \label{baseadaptada} Let $(V,\esc{\cdot,\cdot})$ be a finite-dimensional real vector space $V$ with a symmetric nondegenerate bilinear form $\esc{\cdot,\cdot}\colon V\times V\to\mathbb{R}$.
Assume
that $V=\oplus_{i\in I}V_i$ is the direct sum of   subspaces
in such a way that  for each $i\in I$ there is a unique $j\in I$
such that $\esc{V_i,V_j}\ne 0$. Then  there is a basis
$B=\{u_1,v_1,\ldots, u_r,v_r,z_1,\ldots ,z_q\}$ of $V$ such that
\begin{itemize}
\item $B\subset\cup_i V_i$;
\item $\esc{z_i,z_i}\in\{\pm1\}$, $\esc{u_i,v_i}=1$;
\item Any other inner product of elements in $B$ is zero.
\end{itemize}
\end{lm}

\begin{re}\label{re_ortogonales}
{\rm
If $\Gamma:L=\oplus_{g\in G}L_g$ is a $G$-grading on a   real Lie algebra and $k\colon L\times L\to \mathbb{R}$ is the Killing form of $L$, then $k(L_g,L_h)=0$ if $g+h\ne e$,
because for such $g$ and $h$,  we have $(\ad L_g\ad L_h(L_k))\cap L_k=0$ for all $k$.
If besides  $L$ is semisimple, then $k$ is nondegenerate, so that $k(L_g,L_{-g})\ne0$ for all $g\in\supp(\Gamma)$.
}
\end{re}

\begin{pr}\label{lemaNo}
Let $S$ be a complex simple Lie algebra and $L$ a real form of $S$. Suppose that $L$ inherits    a fine grading on $S$ given by $\Gamma:S=\sum_{g\in G}S_g$. Then
$$
|\sign L - \dim  S_e | \leq \sum_{\stackrel{e\ne g\in G}{2g=e}}\dim  S_g.
$$
\end{pr}

\begin{proof}
The fact that $L$ inherits $\Gamma$ means that $L=\sum_{g\in G}L_g$ with $L_g\oplus \I L_g=S_g$.
By Remark~\ref{re_ortogonales}, we can apply   Lemma~\ref{baseadaptada} to $(L,k)$ for $k$ the Killing form, in order to find   an orthogonal basis for $k$ formed by homogeneous elements of the grading.
Let $B=\{u_1,v_1,\ldots, u_r,v_r,z_1,\ldots ,z_q\}$ be such a basis  ($k({z_j,z_j})=\pm1$, $k({u_i,v_i})=1$). 
Observe that the signature of $k$ coincides with   the signature of the restriction $k\vert_{\sum_{j=1}^q \mathbb{R}z_i}$,
since $k\vert_{<\{u_i,v_i\}>}=  \left(\begin{array}{cc}0&1\\1&0\end{array}\right)$ has zero signature.
Furthermore, the signature of $k$ coincides with the one of the restriction $k\vert_{\sum_{2g=e}L_g}$,
since, taking into account Remark~\ref{re_ortogonales}, $2\deg(z_j)=e$, and for those indices $i$ with $2\deg(u_i)=e$ then $\deg(u_i)=\deg(v_i)$.
Hence
$$
\sign k=\sign k\vert_{\sum_{2g=e}L_g}= \sign k\vert_{L_e} +\sign k\vert_{\sum_{2g=e,g\ne e}L_g}.
$$

Let us first check  that $k\vert_{L_e}$ is positive definite, so that $\sign k\vert_{L_e}=\dim_\mathbb{R} L_e=\dim_\mathbb{C} S_e$. Indeed, we know that $S_e$ is a toral subalgebra formed by semisimple elements (see, for instance, \cite[Corollary~5]{criscanf4}, where the hypothesis of $\Gamma$ being fine is essential). Moreover, if the MAD-group inducing the grading is the direct product of an $l(=\dim_\C S_e)$-dimensional torus $T$ with a finite group, then $T$ is just the torus formed by the automorphisms of $S$ fixing pointwise the subalgebra $S_e$. In particular $S$ is $\mathbb{Z}^l$-graded (a coarsening of $\Gamma$) and $L$ inherits this $\mathbb{Z}^l$-grading. If $L=\sum_{(n_1,\dots,n_l)\in\mathbb{Z}^l}L_{(n_1,\dots,n_l)}$ is such a grading, and $d_i\in\der(L)$ is the derivation defined by $d_i\vert_{L_{(n_1,\dots,n_l)}}=n_i\id$, take $h_i\in L$ such that $d_i=\ad(h_i)$ and note that $  L_e=\span{\{h_1,\dots,h_l\}}$. Now, any $0\ne E=\sum_{i=1}^l s_ih_i\in L_e$ satisfies that $k(E,E)=\sum_{(n_1,\dots,n_l)\in\mathbb{Z}^l}(\sum_i s_in_i)^2\dim L_{(n_1,\dots,n_l)}>0$, because
$\ad^2E$ acts in $L_{(n_1,\dots,n_l)} $ with eigenvalue $ (\sum_i s_in_i)^2$.

With all this in mind, it is clear that 
$$\begin{array}{rl}
|\sign L - \dim_\mathbb{C} S_e |=&|{\sign k\vert_{\sum_{2g=e,g\ne e} L_g}}|\\
\le&\dim_\mathbb{R}\sum_{\stackrel{2g=e}{g\ne e}}L_g=\dim_\mathbb{C}\sum_{\stackrel{2g=e}{g\ne e}}S_g.
\end{array}$$

\end{proof}

\begin{co}
If $L$ is a real form of $\e6$ of signature $-26$, then $L$ cannot inherit any of the following gradings:
$\Gamma_1$, $\Gamma_2$,  $\Gamma_3$, $\Gamma_5$, $\Gamma_6$, $\Gamma_9$, $\Gamma_{10}$ and $\Gamma_{14}$.
\end{co}

\begin{proof}
Let us check first the values in the next table: 
\begin{center}
 { 
 \begin{tabular}{c||
 cccccccc|}

 &$\ \Gamma_1\ $&$\ \Gamma_2\ $&$\ \Gamma_{3}\ $&$\ \Gamma_{5}\ $&
 $\ \Gamma_{6}\ $&$\ \Gamma_{9}\ $&$\ \Gamma_{10}\ $&$\ \Gamma_{14}\ $\\
 \hline\hline \vrule width 0pt height 14pt
 $\dim S_e$&$ 0$&$2$&$0$&$6$&$4$&$0$&$2$&$0$
 \vrule width 0pt height 5pt\\
 \vrule width 0pt height 5pt
 $\sum_{e\ne g\in G,2g=e}\dim S_g$&$0$&$0$&$14$&$0$&$2$&$0$&$16$&$14$
 \vrule width 0pt height 2pt  
 \vrule width 0pt height 2pt \\
 \hline   
  \end{tabular}  
 }
 \end{center}  \medskip
 Then, by applying Proposition~\ref{lemaNo}, any real form inheriting some of these gradings would have signature at least $-14$. (Moreover, 
 $\Gamma_1$ and $\Gamma_9$ are not inherited by any real form of $\e6$; if there is a real form that inherits $\Gamma_2$, necessarily it would have signature $2$; and $\sig{14}$   does not inherit $\Gamma_6$ either.)\smallskip

 Note that the first row is a direct consequence of \cite[Corollary 5]{criscanf4}, which asserts that the dimension of the homogeneous component corresponding to the neutral element coincides, for a fine grading on a simple Lie algebra over $\C$, with the dimension of the quasitorus producing the grading. 
 
 For the second row, we need a case by case  study of $d:=\sum_{e\ne g\in G,2g=e}\dim S_g$. Here we use the descriptions of the gradings  as well as the data about the size  of the homogeneous components provided in \cite{e6}. The universal   group $G_i$ of each grading $\Gamma_i$ was recalled in Table~\ref{tablafinasdee6}.
 \begin{itemize}
 \item[$\Gamma_1$)] The group $G_1=\Z_3^4$ does not have order two elements. The same happens to $G_2=\Z^2\times\Z_3^2$ and to $G_5=\Z^6$. Hence $d=0$ in the three cases.
 \item[$\Gamma_3$)] For any $e\ne g\in\Z_2^3$, the homogeneous component $L_{(\bar0,\bar0,g)}$ has dimension 2, so that $d=2\cdot7$.
 \item[$\Gamma_6$)] Note that $d=\dim L_{(0,0,0,0,\bar1)}=2$, since this homogeneous component coincides with the set of zero trace diagonal elements in the Albert algebra.  
 \item[$\Gamma_9$)] Now $d=\dim L_{(\bar0,\bar0,\bar0, \bar1)}=0$, since the $\Z_3^3$-grading on the Albert algebra has trivial neutral component. 
 \item[$\Gamma_{10}$)] This grading, of type $(60,7,0,1)$, has 7 two-dimensional components corresponding all of them to elements $g\in G_{10}=\Z^2\times\Z_2^3$ satisfying $2g=e$ (one of them is the own $e$). The other component corresponding to an order two element is precisely the only one of dimension 4. Thus $d=6\cdot2+4=16$.
 \item[$\Gamma_{14}$)] All the seven involved homogeneous components have dimension 2, so that  this time $d=14$.
 \end{itemize}
\end{proof}

We cannot apply this technique to $\Gamma_{12}$ and $\Gamma_{13}$   since in both cases $-26\in[\,\dim S_e-c,\dim S_e+c\,]$, but our purpose is still to prove that these gradings are not inherited by $\sig{26}$. In the first case, a classical result due to Cheng provides a tool.

\begin{pr}  ( \cite[Theorem~3]{Cheng})
A simple Lie algebra $L$ admits a $\mathbb{Z}$-grading of the second kind (that is, $L=L_{-2}\oplus L_{-1}\oplus L_{0}\oplus L_{1}\oplus L_{2}  $
with $\dim L_2=1$) if and only if there is a long root $\alpha$ of the restricted root system such that the multiplicity
$m_{\bar \alpha}=1$. 
\end{pr}

Before applying this proposition, recall briefly some basic concepts about restricted roots. If $L=L_{\bar0}\oplus L_{\bar1}$ is a Cartan decomposition and $\mathfrak{a}$ is any maximal abelian subspace of $L_{\bar1}$, then $0\ne\lambda\in\mathfrak{a}^*$
is called a \emph{restricted root} if $L_\lambda=\{x\in L\mid [h,x]=\lambda(h)x\ \forall h\in \mathfrak{a}\}$ is nonzero. The \emph{multiplicity} $m_\lambda$ is defined  as  the dimension of $L_\lambda$. If we take any maximal abelian subalgebra $\h$  of $L$ containing $\mathfrak{a}$, then $\h^\C$ is a Cartan subalgebra of $S$, and, if $\Delta$ denotes the root system of $S$ relative to $\h^\C$, the restricted roots are exactly the nonzero restrictions of roots to $\mathfrak{a}\subset\h^\C$. The Satake
  diagram of $L$ \cite{Satake} is the Dynkin diagram of $S$ where the nodes corresponding to simple roots with restriction zero to $\mathfrak{a}$ are painted in black. The Satake diagram 
  of $\sig{26}$ is 
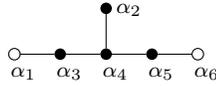
\begin{center}{
\begin{picture}(25,5)(4,-0.5)
\put(5,0){\circle{1}} \put(9,0){\circle*{1}} \put(13,0){\circle*{1}}
\put(17,0){\circle*{1}} \put(21,0){\circle{1}}
\put(13,4){\circle*{1}}
\put(5.5,0){\line(1,0){3}}
\put(9.5,0){\line(1,0){3}} \put(13.5,0){\line(1,0){3}}  
\put(17.5,0){\line(1,0){3}}
\put(13,0.5){\line(0,1){3}}
\put(4.7,-2){$\scriptstyle \alpha_1$} \put(8.7,-2){$\scriptstyle \alpha_3$}
\put(12.7,-2){$\scriptstyle \alpha_4$} \put(16.7,-2){$\scriptstyle \alpha_5$}
\put(20.7,-2){$\scriptstyle \alpha_6$} \put(13.9,3.6){$\scriptstyle \alpha_2$}
\end{picture}
}\end{center}\vskip0.4cm
and the set of restricted roots is just $\Sigma=\pm\{\bar\a_1,\bar\a_6,\overline{\a_1+\a_6}\}$ ($\bar\a$ means $\a\vert_{\mathfrak{a}}$), each of them with multiplicity equal to 8 (\cite[Table~VI]{Helgason} or \cite[Section~5]{Satakes}). Consequently, in spite of being $\Z$-graded, the real form $\sig{26}$ does not admit any $\mathbb{Z}$-grading of the second kind. Now note that the (unique) $\Z$-grading  on $S=\e6$ obtained as a coarsening of $\Gamma_{12}$ is just a $\mathbb{Z}$-grading of the second kind (many details about this $\Z$-grading are developed in \cite[\S3.4]{Weyle6}), with $\dim S_{\pm2}=1$, $\dim S_{\pm1}=20$ and $\dim S_0=36$ (isomorphic to $\mathfrak{gl}_\C(6)$). Consequently, it is impossible that $\sig{26}$ inherits $\Gamma_{12}$.\medskip

Finally we apply the techniques in Section~\ref{sec_gradreales} to study the last grading, $\Gamma_{13}$.

\begin{pr}\label{casitodasheredan13}
The  real forms   of $\e6$ inheriting the $\Z_2^7$-grading $\Gamma_{13}$ are just $\mathfrak{e}_{6,6}$, $\mathfrak{e}_{6,2}$, $\mathfrak{e}_{6,-14}$ and $\mathfrak{e}_{6,-78}$.
\end{pr}

\begin{proof}
Let $\mathcal{B}= \{h_j,e_\alpha,f_\alpha\mid j=1,\dots,6, \alpha\in\Delta^+\} $ be the basis chosen as in Section~\ref{sec_prelimirealforms},
where $\Delta$ is a root system relative to some Cartan subalgebra of $S=\e6$ and $\{\a_1,\dots,\a_6\}$ is a fixed basis of $\Delta$. In particular, $L=\sum_{b\in\mathcal B}\R b$ is a split real form of $S$  with related conjugation $\sigma_0$.
Let $T=\{t_{s_1, \dots, s_6}\mid s_i\in\C^{\ast}\}$ be the torus of $\aut (S)$ producing the Cartan grading, that is, the element $t_{s_1, \dots, s_6}$ acts in the root space $S_\a$, for $\alpha=\sum_{i=1}^6a_i\a_i$, with eigenvalue $s_1^{a_1}\dots s_6^{a_6}$. If $t=t_{s_1, \dots, s_6}\in T$ has order two ($s_i=\pm1$), then 
$t$ acts with the same eigenvalue in $S_{\alpha}$ as in $S_{-\alpha}$.
Consider as in Section~\ref{sec_prelimirealforms} the involutive automorphism $\omega\in\aut(S)$ determined by 
$$
\omega(e_{\alpha_i})=-f_{\alpha_i},\quad
\omega(f_{\alpha_i})=-e_{\alpha_i},\quad
\omega(h_{\alpha_i})=-h_{\alpha_i}.
$$
The  MAD-group of $\aut(S)$ given by
\begin{equation*}
Q=   \{t,\omega t \mid t\in T, \; t^2=1\}
\end{equation*}
is just the MAD-group producing the $\Z_2^7$-grading $\Gamma_{13}$.
 Note that the $\C$-basis of $S=\e6$ given by $\mathcal B'=\{h_j, e_{\alpha}+f_{\alpha}, e_{\alpha}- f_{\alpha} \mid  j=1,\dots, 6,\alpha \in \Delta^+\}$  is homogeneous for $\Gamma_{13}$ (the elements in $\mathcal B'$ are simultaneous eigenvectors for all the elements in $Q$). Moreover,  $\mathcal B'$ is contained in $L=S^{\sigma_0}$. The fact of  $L $ having   a homogeneous $\R$-basis   implies (Proposition~\ref{pr_metodobasereal}) that  the real form $L $ inherits the grading produced by $Q$. Then, by Proposition~\ref{pr_nuestrometodo},
  \begin{equation*} 
  \{\sigma_0 q \mid q\in Q,\, (\sigma_0 q)^2=\id\}= \{\sigma_0 q \mid q\in Q\}
 \end{equation*}
 is exactly the set of conjugations whose related real forms inherit $\Gamma_{13}$ ($\sigma_0$ commutes with any $q$, so that $\sigma_0q$ has order 2 in all the cases).
Let us  analyze the possible signatures related to this set by computing the isomorphism class of the automorphism $\Phi([\sigma_0 q])$ for each $q\in Q$. 
 As $\omega$ commutes with $\sigma_0$ and with all the elements $q\in Q$, and $ \sigma_0 \omega$   is a compact conjugation, then
$
\Phi([\sigma_0 q])=[\omega q]
$
for any $q\in Q$. Recalling the elements in $Q$,
\begin{itemize}
\item $\Phi([\sigma_0 \omega])=[\omega^2]=[\id]$, so that    the compact real form inherits $\Gamma_{13}$.
\item $\Phi([\sigma_0 \omega t])=[t]$. As every inner automorphism is conjugated to one in the maximal torus, here the two types of order two inner automorphisms appear. Therefore, both the real forms $\mathfrak e_{6,2}$ and $\sig{14}$ inherit $\Gamma_{13}$. 
\item $\Phi([\sigma_0 t])= [\omega t]$. But the     outer automorphism $\omega t$   fixes a subalgebra isomorphic to $\c4$ independently of the chosen order two element $t\in T$, that is, $\dim\fix (\omega t)\ne 52$ for any $t$. Indeed, for $t=t_{s_1, \dots, s_6} \in T\cap Q$ ($s_i\in\{\pm1\}$), denote $\Delta^{t,+}=\{\a=\sum a_i\a_i\in\Delta\mid s_1^{a_1}\dots s_6^{a_6}=1\}$ and 
$\Delta^{t,-}=\{\a=\sum a_i\a_i\in\Delta\mid s_1^{a_1}\dots s_6^{a_6}=-1\}=\Delta\setminus\Delta^{t,+}$.  The subalgebra fixed by $\omega t$ is precisely
$$
\langle  \{e_{\alpha}+f_{\alpha}\mid \alpha\in \Delta^{t,+}\cap\Delta^+\}\rangle\oplus 
\langle  \{e_{\alpha}-f_{\alpha}\mid \alpha\in\Delta^{t,-}\cap\Delta^+\}\rangle,
$$
whose dimension is the cardinal of $(\Delta^{t,+}\cap\Delta^+)\cup(\Delta^{t,-}\cap\Delta^+)=\Delta^+$, which is (always) $36$.
 This implies that the split real form $\mathfrak e_{6,6}$ inherits $\Gamma_{13}$ and that this is not the case for $\sig{26}$.
\end{itemize}
\end{proof}

\subsection{Conclusions about the gradings on $\sig{26}$.}

From all the above, we have proved our main result:

\begin{te}
There are exactly 4 fine gradings on $\e6$ producing fine gradings on $\sig{26}$, namely, the $\Z_2^6$-grading $\Gamma_7$,  the $\Z\times\Z_2^4$-grading  $\Gamma_8$, the $\Z^2\times\Z_2^3$-grading  (of inner type) $\Gamma_4$ and the $\Z_4\times\Z_2^4$-grading  $\Gamma_{11}$.
\end{te}

There are some open questions on the study of gradings on $\sig{26}$ which have still to be studied:
\begin{itemize}
\item[a)]  There could exist a fine grading  on $\sig{26}$ whose complexification were not a fine grading on $\e6$.

\item[b)] There could be two fine gradings on $\sig{26}$ not isomorphic but with isomorphic complexification.
\end{itemize}

Up to now, it is not known if   the situations described above could happen: none of the real forms of $\frak{g}_2$ and $\frak{f}_4$ have fine gradings whose complexification is not fine, and neither have they nonisomorphic fine gradings with isomorphic complexification \cite{reales}. But there are examples of nonisomorphic  (not fine)  gradings with isomorphic complexification \cite[Remark 1]{reales}. 
The only work containing a complete study (in the above sense) of fine gradings on a concrete real Lie algebra    is \cite{reales}.
In order to develop a similar theory for the real forms of the Lie algebra $\e6$, we have completed the description of the Weyl groups of the fine gradings on the (complex) Lie algebra $\e6$ \cite{Weyle6}, because the study of item b) is equivalent to the classification of the number of orbits \cite[Proposition~5]{reales} of the action of the normalizer of the MAD-group producing the (complexified) grading on the set of conjugations whose related real forms inherit the grading.


\end{document}